\documentclass{amsart}

\usepackage[english]{babel}
\usepackage{amsmath,amssymb,amsfonts,amsthm}

\usepackage[centering]{geometry}
\usepackage{color}
\usepackage[utf8]{inputenc}

\usepackage{hyperref}
\hypersetup{colorlinks   = true,
  		    citecolor    = green,
   		    linkcolor    = blue}
\usepackage{cleveref}

\newcommand{\metric}[2]{\ensuremath{\left\langle #1, #2\right\rangle_2}}  %% command for the metric  < , >_2

\newcommand{\R}{\ensuremath{\mathbb{R}}}
\newcommand{\C}{\ensuremath{\mathbb{C}}}
\newcommand{\Cc}{\mathbb{C}^{n+2}_2}
\newcommand{\CH}{\mathbb{C}H^{n+1}_1}
\newcommand{\V}{{V^{*}}^{2n+1}_1}
\newcommand{\Q}{{Q^{*}}^n}
\renewcommand{\H}{H^{2n+3}_3}

\newcommand{\lccQ}{\nabla^{{Q^*}^n}}

\DeclareMathOperator{\re}{Re}    %% Real part
\DeclareMathOperator{\im}{Im}    %% Imaginary part

\DeclareMathOperator{\spanned}{span}

\newtheorem{theorem}{Theorem}[section]   %% definition of theorem environment
\newtheorem*{theorem*}{Theorem}          %% a theorem environment without numbering
\newtheorem{lemma}[theorem]{Lemma}
\newtheorem{proposition}[theorem]{Proposition}

\theoremstyle{definition}
\newtheorem{definition}[theorem]{Definition}
\newtheorem{corollary}[theorem]{Corollary}
\newtheorem{example}{Example}
\newtheorem{choice}{Choice}
\newtheorem*{remark}{Remark}

\title[Lagrangian submanifolds of the complex hyperbolic quadric]{Lagrangian submanifolds \\ of the complex hyperbolic quadric}

\author{Joeri Van der Veken}
\author{Anne Wijffels}
\address{KU Leuven, Department of Mathematics, Celestijnenlaan 200B - Box 2400, 3001 Leuven, Belgium}
\email{joeri.vanderveken@kuleuven.be}
\email{anne.wijffels@kuleuven.be}
\thanks{The first author is supported by the Excellence Of Science project G0H4518N of the Belgian government and both authors are supported by project 3E160361 of the KU Leuven Research Fund.}


\begin{document}

\begin{abstract}
We consider the complex hyperbolic quadric ${Q^*}^n$ as a complex hypersurface of complex anti-de Sitter space. Shape operators of this submanifold give rise to a family of local almost product structures on ${Q^*}^n$, which are then used to define local angle functions on any Lagrangian submanifold of ${Q^*}^n$. We prove that a Lagrangian immersion into ${Q^*}^n$ can be seen as the Gauss map of a spacelike hypersurface of (real) anti-de Sitter space and relate the angle functions to the principal curvatures of this hypersurface. We also give a formula relating the mean curvature of the Lagrangian immersion to these principal curvatures. The theorems are illustrated with several examples of spacelike hypersurfaces of anti-de Sitter space and their Gauss maps. Finally, we classify some families of minimal Lagrangian submanifolds of ${Q^*}^n$: those with parallel second fundamental form and those for which the induced sectional curvature is constant. In both cases, the Lagrangian submanifold is forced to be totally geodesic.  
\end{abstract}

\maketitle

%%%%%%%%%%%%%%%%%%%%%%
\section{Introduction}
%%%%%%%%%%%%%%%%%%%%%%

The complex quadric $Q^n$ is a K\"ahler-Einstein manifold, which can be seen in several different ways, for example as a complex hypersurface of the complex projective space $\mathbb C P^{n+1}$, as the Grassmannian manifold of oriented $2$-planes in $\R^{n+2}$ or as the homogeneous space
\[Q^n = \frac{\mathrm{SO}(2+m)}{\mathrm{SO}(2) \times \mathrm{SO}(m)}.\]
Minimal Lagrangian immersions into $Q^n$ were studied for example in \cite{MaOhnita2014}, \cite{MaOhnita2015}, \cite{IriyehMaMiyaokaOhnita2016} and \cite{LMVdVVW}, by identifying them with Gauss maps of isoparametric hypersurfaces of the unit sphere $S^{n+1}(1)$. The relation between the geometric invariants of (not necessarily minimal) Lagrangian submanifolds of $Q^n$ and (not neccesarily isoparametric) hypersurfaces of $S^{n+1}(1)$ was stated in full generality in \cite{VW}.
	
In the present paper, we study Lagrangian submanifolds of the \textit{complex hyperbolic quadric}. This is the homogeneous space 
\[\Q= \frac{\mathrm{SO}^{0}(2,m)}{\mathrm{SO}(2) \times \mathrm{SO}(m)}, \]
which can be identified with the Grassmannian manifold of negative definite oriented $2$-planes in the indefinite vector space $\R^{n+2}_2$. It is known from \cite{MontielRomero} that we can see $\Q$ as a complex hypersurface of the complex anti-de Sitter space $\mathbb C H^{n+1}_1$. We will discuss this immersion in more detail in Section~\ref{sec:1}. In particular, we explain how the induced geometric structures make $\Q$ into a homogeneous (Riemannian) K\"ahler-Einstein manifold, carrying a family of local almost product structures. The complex hyperbolic quadric ${Q^{*}}^1$ of complex dimension $1$ is isometric to the hyperbolic plane $H^2(-2)$ of constant Gaussian curvature $-2$, whereas the complex hyperbolic quadric ${Q^{*}}^2$ of complex dimension $2$ is isometric to the Riemannian product $H^2(-4)\times H^2(-4)$. The latter isometry follows for example from \cite{jensen}, since ${Q^{*}}^2$ is a real $4$-dimensional Einstein manifold whose minimal sectional curvature is $-4$, while it does not have constant sectional curvature, nor constant holomorphic sectional curvature. 

Submanifolds of the complex hyperbolic quadric, in particular curves and real hypersurfaces, have recently attracted attention of several geometers, see for example \cite{Suh2016}, \cite{KimSuh2019} and \cite{KleinSuh2019}. In the present paper, we deal with Lagrangian submanifolds of $\Q$. Inspired by the situation in $Q^n$, where Lagrangian immersions correspond to Gauss maps of spherical hypersurfaces, we define the Gauss map of a spacelike hypersurface of the anti-de Sitter space $H^{n+1}_1(-1)$ as a map into $\Q$, which turns out to be a Lagrangian immersion. Conversely, Theorem~\ref{theo:LagrangiansandGaussmaps} states that one can see any Lagrangian submanifold of $\Q$, at least locally, as the Gaussian image of a hypersurface of $H^{n+1}_1(-1)$. This hypersurface will not be unique, as parallel hypersurfaces have the same Gauss maps. Nevertheless, Theorem~\ref{theo:LagrangiansandGaussmaps} also gives a relation between combinations of the principal curvatures of a hypersurface, which do not change when replacing the hypersurface by a parallel hypersurface, and the so-called \textit{angle functions} of the Gaussian image. The angle functions will be fundamental invariants of Lagrangian submanifolds of $\Q$. Theorem~\ref{theo:Palmerformula} gives an analogue of a formula proven in \cite{Palmer1997} for the case of $Q^n$ in the case of $\Q$: it expresses the mean curvature of a Lagrangian submanifold of $\Q$ in terms of the principal curvatures of a corrsponding hypersurface of $H^{n+1}_1(-1)$. We provide several examples of spacelike hypersurfaces of anti-de Sitter space and their Gauss maps to illustrate the theorems.

In the final two sections of the paper we classify two special classes of minimal Lagrangian submanifolds of $\Q$, namely those with parallel second fundamental form (Theorem \ref{theo:parallel_sff}) and those for which the induced sectional curvature is constant (Theorem \ref{theo:csc}). In both cases, the Lagrangian immersions have to be totally geodesic and they correspond to Gauss maps of well-known families of spacelike hypersurfaces of $H^{n+1}_1(-1)$.

%%%%%%%%%%%%%%%%%%%%%%%%%%%%%%%%%%%%%%%%%%%%%%%%%%%%%%%%%%%%%%%%%%%%%%%%%%%
\section{The geometry of the complex hyperbolic quadric $\Q$} \label{sec:1}
%%%%%%%%%%%%%%%%%%%%%%%%%%%%%%%%%%%%%%%%%%%%%%%%%%%%%%%%%%%%%%%%%%%%%%%%%%%

We first introduce some notation which will be used throughout the paper. For integers $d$ and $i$ satisfying $0 \leq i \leq d$, the pseudo-Euclidean space of dimension $d$ and index $i$, denoted as $\mathbb R^d_i$, is $\mathbb R^d$ equipped with the metric $\langle (x_1,\ldots,x_d),(y_1,\ldots,y_d)\rangle_i = -x_1y_1 - \ldots - x_iy_i + x_{i+1}y_{i+1} + \ldots + x_dy_d$. The pseudo-hyperbolic space $H^d_i(c)$ of dimension $d$, index $i$ and constant sectional curvature $c<0$ is then
$$ H^d_i(c) = \left\{ x \in \mathbb R^{d+1}_{i+1} \ \left| \ \langle x,x \rangle_{i+1} = \frac{1}{c} \right. \right\}. $$
In particular, for $i=0$, we have the $d$-dimensional hyperbolic space $H^d_0(c) = H^d(c)$ and for $i=1$, we have the $d$-dimensional anti-de Sitter space $H^d_1(c)$.

To define the complex hyperbolic quadric, we first consider the complex space $\C^{n+2}_2$, which is $\mathbb C^{n+2}$ with the metric $\langle\!\langle(z_0,\ldots,z_{n+1}),(w_0,\ldots,w_{n+1})\rangle\!\rangle_2 =\re( -z_0\bar{w}_0- z_1\bar{w}_1 + \cdots + z_{n+1}\bar{w}_{n+1})$. Remark that, under the natural identification of $\mathbb C$ with $\mathbb R^2$, we have $\mathbb C^{n+2}_2 \approx \mathbb R^{2n+4}_4$. We now define the complex anti-de Sitter space $\CH$ as the set of all complex 1-dimensional subspaces of $\Cc$, on which $\langle\!\langle\cdot,\cdot\rangle\!\rangle_2$ is negative definite. If we equip $\CH$ with a natural differential structure, the projection map $\{ z \in \Cc \ | \ \langle\!\langle z,z \rangle\!\rangle_2 < 0 \} \to \CH : z \mapsto [z]$, where $[z]$ denotes the complex line generated by $z$, is a submersion. This remains true if we restrict the map to $\H(-1) \subset \mathbb R^{2n+4}_4$. (Note that, under the identification $\mathbb R^{2n+4}_4 \!\approx\! \mathbb C^{n+2}_2$, we have $\H(-1) \subset \{ z \in \Cc \ | \ \langle\!\langle z,z \rangle\!\rangle_2 < 0 \}$.) From now on, we refer to the map  
\[ \pi: \H(-1) \subset \Cc\rightarrow \CH : z\mapsto [z] \]
as the \textit{Hopf fibration}. If we equip $\CH$ with a pseudo-Riemannian metric $g$ such that $\pi$ becomes a pseudo-Riemannian submersion, it turns out that this metric has to have complex index $1$ (as already suggested by the notation $\CH$) and constant holomorphic sectional curvature $-4$. We denote $\CH(-4) = (\CH,g)$. 

Note that the fibers of the Hopf fibration are timelike curves. In fact, for any $z \in \H(-1)$, we have that $\pi^{-1}([z])=\{e^{it}z \ | \ t \in \R\}$ and $\ker(d\pi)_z = \spanned\{\xi_z\}$, where $\xi_z= iz$ satisfies $\langle\!\langle \xi_z,\xi_z \rangle\!\rangle_2 = \langle\!\langle z,z \rangle\!\rangle_2 = -1$. The complex structure $J$ on $\CH(-4)$ is induced by multiplication by $i$ on $T\H(-1)$ and $(\CH(-4), g, J)$ is a K\"ahler manifold.
	
We define the \textit{complex hyperbolic quadric} of dimension $n$ as
\begin{equation*}
\Q=\left\{\:[(z_0,z_1,\dots,z_{n+1})]\in\CH(-4)\:|\: -z_0^2 -z_1^2 + \cdots + z_{n+1}^2 =0\:\right\}.
\end{equation*}
If $\Q$ is equipped with the induced metric $g|_{\Q}$, which we will again denote by $g$, and the induced almost complex structure $J|_{\Q}$, which we will again denote by $J$, then $(\Q, g,J)$ is a K\"ahler manifold itself.
	
The inverse image of $\Q$ under the Hopf fibration is given by
\begin{equation} \label{eq:defV}
\V=\left\{ u+iv \:\left|\: u,v \in\R^{n+2}_2, \metric{u}{u}= \metric{v}{v}= - \frac{1}{2}, \metric{u}{v} = 0 \right.\right\} \subset \H(-1). 
\end{equation}
Remark that $\V$ is a submanifold of $\H(-1)$ of real dimension $2n+1$. Its real index is $1$, since the normal space to $\V$ as a submanifold of $\H(-1) \subset \mathbb C^{n+2}_2$ at $z$ is spanned by the orthogonal vectors $\bar{z}$ and $i\bar{z}$, satisfying $\langle\!\langle \bar z, \bar z \rangle\!\rangle_2 = \langle\!\langle i\bar z, i\bar z \rangle\!\rangle_2 = -1$. This implies that $\Q$ is indeed a Riemannian submanifold of $\CH(-4)$, where the normal space $T_{[z]}^{\perp}\Q$ is spanned by $(d\pi)_z(\bar{z})$ and $J(d\pi)_z(\bar{z})=(d\pi)_z(i\bar{z})$. 
	
From \cite{MontielRomero} we have the following.
%Another way to see this is using the Hopf-fibration, since for horizontal tangent vector fields on $\V$ it holds that $\widetilde{A}^2=\id$, where $\widetilde{A}$ is the lift of $A$ meaning it is itself a shape operator of the immersion of $\V$ in $\H(-1)$. 
	
\begin{lemma}\label{lem:PropertiesOfA}
Let $\mathcal{A}$ be the set of shape operators of $\Q$ as a submanifold of $\CH(-4)$, associated to (locally defined) unit normal vector fields. Then any $A\in\mathcal{A}$ is symmetric, involutive and anti-commutes with $J$. In particular, $A$ is an almost product structure on (an open subset of) $\Q$.
\end{lemma}
If $A \in \mathcal A$ is the shape operator associated to a unit normal vector field $\zeta$ along $\Q$ in $\CH(-4)$, the formula of Weingarten reads
\begin{equation*}
\nabla^{\CH(-4)}_X\zeta= -AX +s(X)J\zeta
\end{equation*}
for any tangent vector $X$ to $\Q$, where $\nabla^{\CH(-4)}$ is the Levi Civita connection of $\CH(-4)$ and $s$ is a one-form on $\Q$. Note that this one-form depends on the choice of $\zeta$ and hence on the choice of almost product structure $A$. Using the same reasoning as in \cite{MontielRomero,smyth}, we obtain 
\begin{equation*}
\nabla^{\Q}_X A= s(X) JA \label{eq:CovariantDerivativeOfA}
\end{equation*}
for any tangent vector $X$ to $\Q$, where $\nabla^{\Q}$ denotes the Levi-Civita connection on $\Q$.  
	
The equation of Gauss for $\Q$ as a submanifold of $\CH(-4)$ yields the following expression for the Riemann-Christoffel curvature tensor of $\Q$:
	\begin{equation}\label{eq:RiemCurQn*}
	\begin{aligned}
	R^{\Q}(X,Y)Z = -g&(Y,Z)X +g(X,Z)Y-g(X,JZ)JY +g(Y,JZ)JX - 2g(X,JY)JZ \\
	&- g(AY,Z)AX + g(AX,Z)AY - g(JAY,Z) JAX + g(JAX,Z) JAY,
	\end{aligned}
	\end{equation}
where $A$ is any element of $\mathcal{A}$. In particular, it is an almost product structure on $\Q$ which anti-commutes with $J$ by Lemma \ref{lem:PropertiesOfA}. Notice that the Riemann-Christoffel curvature tensor of $\Q$ is exactly opposite to the Riemann-Christoffel curvature tensor of $Q^n$, which explains why Lagrangian submanifolds of these two spaces can be studied by similar methods, as we will see in the next section.

%%%%%%%%%%%%%%%%%%%%%%%%%%%%%%%%%%%%%%%%%%%%%%%%%%%%%%%
\section{Lagrangian submanifolds of $\Q$} \label{sec:2}
%%%%%%%%%%%%%%%%%%%%%%%%%%%%%%%%%%%%%%%%%%%%%%%%%%%%%%%
	
Consider an immersion $f:M^n\rightarrow \Q$ from a manifold of real dimension $n$ into the complex hyperbolic quadric of complex dimension $n$. If no confusion is possible, we will identify $(df)_p(T_pM^n)$ with $T_pM^n$ for every $p\in M^n$. Moreover, we will denote the metric on $M^n$, induced from the metric $g$ on $\Q$, again by $g$. As usual in complex geometry, we say that $f$ is \textit{Lagrangian} if $J$ maps the tangent space to $M^n$ at a point into the normal space to $M^n$ at that point and vice versa.
	
Fixing an almost product structure $A\in \mathcal{A}$ on $\Q$, we can define, at any point $p$ of a Lagrangian submanifold $M^n$ of $\Q$, two endomorphisms $B$ and $C$ of $T_pM^n$ by putting 
\begin{equation} 
AX=BX-JCX   \label{eq:DefinitionBandC}
\end{equation}
for all $X\in T_pM^n$, i.e., $BX$ is the component of $AX$ tangent to $M^n$ and $CX$ is the image under $J$ of the component of $AX$ normal to $M^n$. The same construction was done in \cite{LMVdVVW} for $Q^n$ and, in a similar way, the following result easily follows from Lemma \ref{lem:PropertiesOfA}.
	
\begin{lemma}\label{lem:DefinitionBandC}
Let $p$ be a point of a Lagrangian submanifold $M^n$ of $\Q$. The endomorphisms $B$ and $C$ of $T_pM^n$, defined by \eqref{eq:DefinitionBandC}, are symmetric, they commute and they satisfy $B^2+C^2=\operatorname{id}_{T_pM^n}$.
\end{lemma}
%
%	\begin{proof}
%		Since $g(BX,Y)=g(AX,Y)$ and $g(CX,Y)=g(JAX,Y)$ for all $X,Y\in T_pM^n$, the endomorphisms $B$ and $C$ are symmetric.
%		
%		Furthermore, we have $X=A^2X=A(BX-JCX)=(B^2+C^2)X+J(BC-CB)X$ for an arbitrary vector $X\in T_pM^n$. Since the first term on the right hand side is tangent to $M^n$ and the secondterm on the right hand side is normal to $M^n$, we get $(B^2+C^2)X=X$ and $(BC-CB)X=0$, which proves the result.
%	\end{proof}
%
Since $\Q$ is Riemannian, it follows that $B$ and $C$ are simultaneously diagonalizable and the sum of the squares of corresponding eigenvalues is equal to 1. Thus, there exists an orthonormal basis $\{e_1,\dots,e_n\}$ of $T_pM^n$ and real numbers $\theta_1,\dots,\theta_n$, defined up to an integer multiple of $\pi$, such that $B e_j = \cos(2\theta_j) e_j$ and $Ce_j = \sin(2\theta_j) e_j$ for $j \in \{1,\dots,n\}$. The factor $2$ is just a choice and we can write both equalities together as
\begin{equation}\label{eq:defBandC}
Ae_j = \cos(2\theta_j)e_j-\sin(2\theta_j)Je_j.
\end{equation}

Working locally, we can look at $B$ and $C$ as symmetric $(1,1)$-tensor fields on $M^n$ which define a local orthonormal frame $\{e_1,\dots,e_n\}$ and local angle functions $\theta_1,\dots,\theta_n$. 
	
%	\begin{lemma}\label{lem:RelationAtoA0}
%		Let $f:M^n \rightarrow \Q$ be a Lagrangian immersion and $A_0,A\in\mathcal{A}$. Then there exists a function $\alpha:M^n\rightarrow \mathbb{R}$ such that, along the image of $f$,
%		\[A=\cos(\varphi)A_0 + \sin(\varphi) JA_0.\]
%		If $\{e_1,\dots,e_n\}$ is a local orthonormal frame such that $A_0e_j=\cos(2\theta^0_j)e_j-\sin(2\theta^0_j)Je_j$ for $j=1,\dots,n$, then $Ae_j=\cos(2\theta_j)e_j-\sin(2\theta_j)Je_j$ for $j=1,\dots,n$, with
%		\[\theta_j=\theta_j^0 -\varphi/2.\]
%	\end{lemma}
%	\begin{proof}
%		Assume that $A_0$ and $A$ are the shape operators associated to the normal vector fields $\zeta_0$ and $\zeta$ respectively. Since $\Q$ is a K\"ahler submanifold of $\mathbb{C}H^{n+1}_1(-4)$, there is a function $\varphi:M^n\rightarrow \mathbb{R}$ such that, at every point of $M^n$, $\zeta=\cos(\varphi)\zeta_0 +\sin(\varphi) J\zeta_0$, which implies that $A=\cos(\varphi)A_0+\sin(\varphi)JA_0$.  
%		Now assume that $A_0e_j=\cos(2\theta_j^0)e_j -\sin(2\theta_j^0)Je_j. $ Then it follows from a straight forward computation that $Ae_j=\cos(2\theta_j^0 -\varphi) e_j-\sin(2\theta_j^0-\varphi)Je_j. $
%	\end{proof}
	
There are different choices for the almost product structure $A\in\mathcal{A}$ on $\Q$. The following lemma, whose counterpart for $Q^n$ was proven in \cite{LMVdVVW}, shows how a change of the almost product structure changes the local angle functions of a Lagrangian submanifold of $\Q$.
\begin{lemma}\label{lem:changeA}
Let $f : M^n \to \Q$ be a Lagrangian immersion and $A_0,A\in\mathcal{A}$. Then there exists a function $\phi: M^n \to \mathbb{R}$ such that, along the image of $f$,
$$ A = \cos\phi \, A_0 + \sin\phi \, JA_0. $$ 
If $\{e_1,\dots, e_n\}$ is a local orthonormal frame with $A_0e_j = \cos(2\theta^0_j)e_j - \sin(2\theta^0_j) Je_j$ for $j\in\{1,\dots, n\}$, then $Ae_j = \cos(2\theta_j)e_j - \sin(2\theta_j)Je_j$ for $j\in\{1,\dots, n\}$, with
$$ \theta_j =\theta^0_j-\frac{\phi}{2}. $$
\end{lemma}

If $h$ is the second fundamental form of the Lagrangian immersion $f:M^n\rightarrow \Q$, we define 
\begin{equation*}
h_{ij}^k = g(h(e_i,e_j),J e_k)
\end{equation*}
for all $i,j,k \in \{1,\dots,n\}$, to be the components of $h$. A fundamental property of Lagrangian submanifolds implies that the components $h_{ij}^k$ are symmetric in the three indices. Furthermore, if $\nabla$ is the induced connection on $M^n$ from the Levi Civita connection $\nabla^{\Q}$, we define its connection one-forms by
\begin{equation*}
\omega_{j}^k(X) = g(\nabla_{X} e_j, e_k)
\end{equation*}
for all $j,k \in \{1,\dots,n\}$ and $X$ tangent to $M^n$. Remark that these one-forms are anti-symmetric in their indices.
	
\begin{proposition} \label{prop:verband_theta_h}
Let $M^n$ be a Lagrangian submanifold of $\Q$ and let $A\in\mathcal{A}$ be an almost product structure on $\Q$. Let $\{e_1,\dots,e_n\}$ be a local orthonormal frame on $M^n$ as constructed above, then the following relations between the angle functions, the components of the second fundamental form and the connection forms hold:
\begin{align}
& e_i(\theta_j)= h_{jj}^i-\frac{s(e_i)}{2},\label{eq:DerivativeOfTheta}\\
& \sin(\theta_j -\theta_k)\omega_{j}^k(e_i) = \cos(\theta_j-\theta_k)h_{ij}^k \label{eq:relation_h_omega}
\end{align}
for all $i,j,k \in \{1,\dots,n\}$ with $j\neq k$.
\end{proposition}

The proof is similar to that of the corresponding result for $Q^n$, which can be found in \cite{LMVdVVW}. The following result easily follows from \eqref{eq:DerivativeOfTheta}.
	
\begin{corollary}\label{cor:MinimalLagrangianImmersion}
Let $f:M^n\rightarrow \Q$ be a minimal Lagrangian immersion for which the sum of the local angle functions is constant. 
%This can for example be achieved by choosing $A\in \mathcal{A}$ such that the sum of the angle functions vanishes modulo $\pi$. 
Then the one-form $s$  associated to $A$ vanishes on tangent vectors to $M^n$. 
%In particular, for all $X$ tangent to $M$, one has $\nabla^{\Q}_X A=0$ and, if $A$ is the shape-operator associated to a normal vector field $\zeta$ along $\Q$, also $\nabla^{\perp}_X \zeta=0$, where $\nabla^{\perp}$ is the normal connection of $\Q$ in $\mathbb{C}H^{n+1}_1(-4)$.
\end{corollary}

\begin{proof}
We choose $\{e_1,\ldots,e_n\}$ and $\theta_1,\ldots,\theta_n$ as above. Since the sum of the local angle functions is constant, \eqref{eq:DerivativeOfTheta} implies
$$ 0 = e_i(\theta_1+\ldots+\theta_n) = h_{11}^i + \ldots + h_{nn}^i - n \frac{s(e_i)}{2} = - n \frac{s(e_i)}{2}, $$
where we used minimality in the last equality. We conclude that $s(e_i)=0$ for all $i \in \{1,\ldots,n\}$, which means that $s$ vanishes on all tangent vectors to $M^n$.
\end{proof}
	
We will now state the equations of Gauss and Codazzi for a Lagrangian submanifold of $\Q$.
	
\begin{proposition}[Equations of Gauss and Codazzi]
Let $f:M^n\rightarrow \Q$ be a Lagrangian immersion with second fundamental form $h$. Define $B$ and $C$ as above for any choice of $A\in \mathcal{A}$. Finally, denote by $R$ the Riemann-Christoffel curvature tensor of $M^n$ and by $\bar{\nabla}$ the connection of Van der Waerden-Bortolotti. Then
\begin{equation}\label{eq:GaussEqLagSubm}
\begin{aligned}g(R(X,Y)Z,W) = -&g(Y,Z)g(X,W) + g(X,Z)g(Y,W) \\
& -g(BY,Z)g(BX,W) + g(BX,Z)g(BY,W)\\
& -g(CY,Z)g(CX,W) + g(CX,Z)g(CY,W)\\
& + g(h(Y,Z),h(X,W)) - g(h(X,Z), h(Y,W))
\end{aligned}
\end{equation}
and
\begin{equation}\label{eq:CodazziEqLagSubm}
\begin{aligned}
(\bar{\nabla}h)(X,Y,Z)-(\bar{\nabla}h) (Y,X,Z) = g&(BY,Z)JCX -g(BX,Z)JCY\\
&-g(CY,Z)JBX +g(CX,Z) JBY
\end{aligned}
\end{equation}
for any vector fields $X,Y,Z$ and $W$ tangent to $M^n$. 
\end{proposition}
\begin{proof}
These follow immediately from the general forms of the equations of Gauss and Codazzi,
\begin{align*}
& g(R(X,Y)Z,W) = g(R^{\Q}(X,Y)Z,W) + g(h(Y,Z),h(X,W)) - g(h(X,Z), h(Y,W)),\\
& (\bar{\nabla}h)(X,Y,Z) - (\bar{\nabla}h)(Y,X,Z) = (R^{\Q}(X,Y)Z)^{\perp},
\end{align*}
where the superscript $\perp$ denotes the component normal to $M^n$, by using (\ref{eq:RiemCurQn*}) and (\ref{eq:DefinitionBandC}). 
\end{proof}

\begin{remark}
Note that the Ricci equation for Lagrangian submanifolds of $\Q$ is equivalent to the Gauss equation. 
\end{remark}
%	\begin{remark}
%		If $\{e_1,\dots,e_n\}$ is the orthonormal frame constructed above and $\theta_1,\dots,\theta_n$ are the angle functions, then it follows from (\ref{eq:GaussEqLagSubm}) and (\ref{eq:DerivativeOfTheta}) that the sectional curvature of the plane spanned by $e_i$ and $e_j$ is given by
%		\begin{equation}\label{eq:SecCurvLagSubm}
%		\begin{aligned}
%		K_{ij} & = g(R(e_i,e_j) e_j,e_i)\\
%		&= -2 \cos^2(\theta_i-\theta_j) + g(h(e_i,e_i),h(e_j,e_j)) - g(h(e_i,e_j),h(e_i,e_j)) \\
%		& =-2\cos^2(\theta_i-\theta_j)+\sum_{k=1}^n \left(h_{jj}^k h_{ii}^k -(h_{ij}^k)^2 \right)\\
%		& = -2\cos^2(\theta_i-\theta_j) +\sum_{k=1}^n\left( \left(\frac{s(e_k)}{2} + e_k(\theta_j)\right)\left(\frac{s(e_k)}{2} + e_k(\theta_i)\right)-(h_{ij}^k)^2 \right)
%		\end{aligned}
%		\end{equation}
%		for any $i,j=1, \dots, n$, with $i\neq j$.
%	\end{remark}

We finish this section by giving two specific choices of an almost product structure $A \in \mathcal A$, adapted to a given Lagrangian submanifold of $\Q$.

\begin{choice}[Choice of $A$ along a Lagrangian submanifold of $\Q$ such that the sum of the angle functions vanishes] \label{ex:choiceA2}
Given a Lagrangian immersion $f : M^n \to \Q$, one can choose $A \in\mathcal{A}$ such that
the associated local angle functions satisfy
\begin{equation}
\theta_1 +\cdots + \theta_n = 0 \mod \pi. \label{eq:sumtheta0}
\end{equation}
Indeed, let $A_0 \in \mathcal{A}$ be an arbitrary almost product structure with associated local angle functions $\theta^0_1, \dots,\theta^0_n$ and put $\phi = 2(\theta^0_1+\cdots+\theta^0_n)/n$. If we choose $A\in\mathcal{A}$ such that $A = \cos\phi \, A_0 + \sin\phi \, JA_0$ along the image of $f$, then it follows from Lemma \ref{lem:changeA} that the local angle functions associated to $A$ satisfy \eqref{eq:sumtheta0}. Note that this implies that the condition of Corollary \ref{cor:MinimalLagrangianImmersion} can always be met by a suitable choice of $A$.
\end{choice}

\begin{choice}[Choice of $A$ along a Lagrangian submanifold of $\Q$ with a given horizontal lift] \label{ex:choiceA}
Assume that both a Lagrangian immersion $f: M^n \to \Q$ and a horizontal lift $\tilde f:~M^n \to \V$ of $f$ are given. It follows from \cite{Reckziegel1985} that any Lagrangian immersion into $\Q$ locally allows such a horizontal lift. If $M^n$ is simply connected, the horizontal lift can be defined globally. Since the normal space to $\V$ in $H^{2n+3}_3(-1) \subset \C^{n+2}_2$ at a point $z$ is the complex span of $\bar z$, one can take $\zeta$, defined by $\zeta_{f(p)}=(d\pi)_{\tilde f(p)}\left(\overline{\tilde{f}(p)}\right)$, as a unit normal vector field to $\Q$ in $\C H^{n+1}_1(-4)$ along the image of $f$. The corresponding shape operator is given by $
A X = -(d\pi)_{\tilde f(p)}\left(\overline{\tilde X}\right)$, 
where $X$ is any vector tangent to $\Q$ at a point $f(p)$ and $\tilde X$ is its horizontal lift to $\tilde f(p)$. In the special case that $v$ is tangent to $M^n$ at a point $p$, we have 
\begin{equation} \label{eq:choiceA}
A(df)_p v = -(d\pi)_{\tilde f(p)}\left(d\overline{\tilde f}\right)_p v. 
\end{equation}
This $A$ can be extended to an element of $\mathcal A$, defined in a neighborhood of $f(M^n)$.
\end{choice}
	
%%%%%%%%%%%%%%%%%%%%%%%%%%%%%%%%%%%%%%%%%%%%%%%%%%%%%%%%%%%%%%%%%%%%%%%%%%%%%%%%%%%%%%%
\section{A Gauss map for spacelike hypersurfaces of anti-de Sitter space} \label{sec:3}
%%%%%%%%%%%%%%%%%%%%%%%%%%%%%%%%%%%%%%%%%%%%%%%%%%%%%%%%%%%%%%%%%%%%%%%%%%%%%%%%%%%%%%%

\subsection{Definition and properties}

We define a notion of Gauss map for spacelike hypersurfaces of the anti-de Sitter space $H^{n+1}_1(-1)$. The following definition is inspired by the definition of Gauss map of a hypersurface of a sphere, which was given for instance in \cite{Palmer1997}.

\begin{definition}
Let $a:M^n\rightarrow H^{n+1}_1(-1) \subset \mathbb{R}^{n+2}_2$ be a spacelike immersion and denote by $b$ a unit normal vector field along this immersion, tangent to $H^{n+1}_1(-1)$. We define the Gauss map
of the hypersurface $a$ by
\begin{equation*}
G:M^n \rightarrow \Q: p\mapsto [a(p)+ib(p)].
\end{equation*}
\end{definition}	
	
Note that $G$ indeed takes values in the complex hyperbolic quadric: for every $p \in M^n$, we have that $\langle a(p),a(p) \rangle_2 = \langle b(p),b(p) \rangle_2 = -1$ and $\langle a(p),b(p) \rangle_2 = 0$, such that $\frac{1}{\sqrt{2}}(a(p)+ib(p)) \in \V$ by \eqref{eq:defV}. This implies that $[a(p)+ib(p)] = [\frac{1}{\sqrt{2}}(a(p)+ib(p))] = \pi(\frac{1}{\sqrt{2}}(a(p)+ib(p)) \in \Q$. 

In the following, we shall refer to
\begin{equation*}
\widetilde{G}: M^n\rightarrow \V: p \mapsto \frac{1}{\sqrt{2}}(a(p)+ib(p))
\end{equation*}
as the \textit{canonical lift} of the Gauss map of a hypersurface $a: M^n \to H_1^{n+1}(-1)$ with a fixed unit normal $b$. Using this lift, we can prove that the Gauss map $G$ is Lagrangian. 

\begin{lemma}\label{lem:lagrangian}
The Gauss map $G:M^n \rightarrow \Q$ of a spacelike hypersurface $a: M^n \to H^{n+1}_1(-1)$ of anti-de Sitter space is a Lagrangian immersion. 
\end{lemma}

\begin{proof}
Let $S$ be the shape operator of the immersion $a$ associated to the unit normal vector field $b$ that was used to construct the Gauss map. Denote by $\{e_1,\dots,e_n\}$ a local orthonormal frame of principal directions of $a$ and by $\lambda_1,\dots,\lambda_n$ the corresponding principal curvatures such that $Se_j=\lambda_j e_j$ for $j \in \{1,\dots n\}$. Then $G$ is a Lagrangian immersion since the canonical lift $\widetilde{G}$ satisfies
\begin{equation}
(d\widetilde{G})e_j = \frac{1}{\sqrt{2}} (1-i\lambda_j) e_j \label{eq:LiftOfBasisUnderGaussMap}
\end{equation}
for all $j \in \{1,\ldots,n\}$, which is perpendicular to $i(d\widetilde{G})e_k$ for all $k \in \{1,\dots,n\}$. 
\end{proof}
\begin{remark}
We can see from \eqref{eq:LiftOfBasisUnderGaussMap} that $(d\widetilde{G})e_j$ is orthogonal to $i\widetilde G$ for all $j=1,\dots,n$. This shows that $\widetilde G$ is horizontal, meaning that $\widetilde G$ is the unique horizontal lift of $G$, up to multiplication with a factor $e^{it}$ for some constant $t\in\mathbb{R}$.
\end{remark}

Furthermore, hypersurfaces of anti-de Sitter space which are parallel to each other have the same Gauss maps.
 
\begin{lemma}\label{lem:parallel}
Let $a:M^n\rightarrow H^{n+1}_1(-1)$ be an immersion and let $a':M^n\to H^{n+1}_1(-1)$ be parallel to $a$, with the same orientation as $a$. Then $a$ and $a'$ have the same Gauss maps.
\end{lemma}

\begin{proof}
Denote by $b$ a unit normal vector field to $a$ and by $b'$ a unit normal vector field to $a'$, both inducing the same orientation on $M^n$. Since $a'$ is parallel to $a$ we have that
\begin{align*}
&a'(p)= (\cos t) \, a(p)+ (\sin t) \, b(p),\\
&b'(p)= -(\sin t) \, a(p) + (\cos t) \, b(p)
\end{align*} 
for some $t\in\R$. We immediately see that 
$$[a'(p)+ib'(p)]=[(\cos t \!-\! i\sin t)a(p)+ (\sin t \!+\! i\cos t)b(p)]=[e^{-it}(a(p)+ib(p))] = [a(p)+ib(p)].$$
\end{proof} 

\subsection{Examples of spacelike hypersurfaces of anti-de Sitter space and their Gauss maps}

We now give several families of examples of spacelike hypersurfaces of $H^{n+1}_1(-1)$ and their Gauss maps. We also compute the principal curvatures of the hypersurfaces and the angle functions of the Gauss maps.

\begin{example} \label{ex:1princcurv}
For any real constant $\alpha$, with $\sin\alpha \neq 0$, the immersion
$$ a_1: H^n(-1) \to H^{n+1}_1(-1) : p \mapsto (\cos\alpha , (\sin\alpha) p) $$
defines a totally umbilical hypersurface of $H^{n+1}_1(-1)$. In the following, we will refer to such an immersion as \emph{a standard embedding of $H^n$ into $H^{n+1}_1(-1)$}. If we choose $b_1(p) = (\sin\alpha,-(\cos\alpha)p)$ as the unit normal vector field, the principal curvatures are $\lambda_1 = \ldots = \lambda_n = \cot\alpha$. The Gauss map of the hypersurface is given by
$$ G_1: H^n(-1) \to \Q : p \mapsto [(\cos\alpha+i\sin\alpha, (\sin\alpha-i\cos\alpha)p)] = [(i,p)]. $$
Remark that $G_1$ is independent of $\alpha$. This can also be seen as a consequence of Lemma \ref{lem:parallel}, since the hypersurfaces in this family are parallel to each other.

Let us compute the local angle functions of $G_1$. If we choose $A \in \mathcal A$ as in Choice \ref{ex:choiceA} using the horizontal lift $\widetilde G_1: H^n(-1) \to \V : p \mapsto \frac{1}{\sqrt 2}(i,p)$, formula \eqref{eq:choiceA} becomes 
$$ A(dG_1)_p v = -(d\pi)_{\widetilde G_1(p)}\left(d\overline{\widetilde G_1}\right)_p v = -(d\pi)_{\widetilde G_1(p)}\left(d\widetilde G_1\right)_p v = -(d(\pi \circ \widetilde G_1))_p v = -(dG_1)_p v, $$
which shows that $B=-\mathrm{id}$ and $C=0$ or, equivalently, $\theta_1=\ldots=\theta_n=\frac{\pi}{2} \mod \pi$. For this choice of $A$, the one-form $s$ vanishes on tangent vectors by Corollary \ref{cor:MinimalLagrangianImmersion}. From Proposition \ref{prop:verband_theta_h}, we then obtain that $G_1$ is totally geodesic. Note that for a general choice of $A \in \mathcal A$, we have
$$ \theta_1=\ldots=\theta_n=\phi \mod\pi $$
for a function $\phi: H^n(-1) \to \mathbb R$. 

\end{example}

\begin{example} \label{ex:2princcurv}
To describe the second family of examples, we use the map 
$$ \psi: \mathbb{R}^{k+1}_1 \! \times \mathbb{R}^{n-k+1}_1 \to \mathbb{R}^{n+2}_2: (p_1,\ldots,p_{k+1},q_1,\ldots,q_{n-k+1}) \mapsto (p_1,q_1,p_2,\dots,p_{k+1},q_2,\dots,q_{n-k+1}). $$
For any real constant $\alpha$, with $\cos\alpha\sin\alpha \neq 0$, the map
$$ a_2: H^k(-1) \times H^{n-k}(-1) \to H^{n+1}_1(-1): (p,q) \mapsto \psi((\cos\alpha)p,(\sin\alpha)q) $$
is an immersion from a product of hyperbolic spaces into anti-de Sitter space. We will refer to such an immersion as \emph{a standard embedding of $H^k \times H^{n-k}$ into $H^{n+1}_1(-1)$}. If we choose the unit normal $b_2(p,q) = \psi((\sin\alpha)p,-(\cos\alpha)q)$, then the principal curvatures of the hypersurface are $\lambda_1=\ldots=\lambda_k=-\tan\alpha$ and $\lambda_{k+1}=\ldots=\lambda_n=\cot\alpha$. Moreover, the Gauss map is given by
$$ G_2: H^k(-1) \times H^{n-k}(-1) \to \Q: (p,q) \mapsto [\psi((\cos\alpha \! + \! i\sin\alpha)p,(\sin\alpha \! - \! i\cos\alpha)q)] = [\psi(ip,q)]. $$
As in the previous example, the Gauss map is independent of $\alpha$, since all hypersurfaces of this family are parallel to each other.

The angle functions of $G_2$ can be computed as in the previous example, using the horizontal lift $\widetilde G_2: H^k(-1) \times H^{n-k}(-1) \to \V : (p,q) \mapsto \frac{1}{\sqrt 2}\psi(ip,q)$: for a vector $v$ tangent to $H^k(-1)$, we find $A(dG_2)_{(p,q)}v = -(dG_2)_{(p,q)}v$, whereas for a vector $w$ tangent to $H^{n-k}(-1)$, we find $A(dG_2)_{(p,q)}w = (dG_2)_{(p,q)}w$. If we choose the orthonormal frame $\{e_1,\ldots,e_n\}$ such that $e_1,\ldots,e_k$ are tangent to $H^k(-1)$ and $e_{k+1},\ldots,e_n$ are tangent to $H^{n-k}(-1)$, then $\theta_1=\ldots=\theta_k=\frac{\pi}{2} \mod\pi$ and $\theta_{k+1}=\ldots=\theta_n=0 \mod\pi$. As in the previous example, it follows from Corollary \ref{cor:MinimalLagrangianImmersion} and Proposition \ref{prop:verband_theta_h} that $G_2$ is totally geodesic. Note that, with respect to an arbitrary $A \in \mathcal A$, 
$$ \theta_1=\ldots=\theta_k = \phi \mod\pi, \qquad \theta_{k+1}=\ldots=\theta_n = \phi + \frac{\pi}{2} \mod\pi $$
for a function $\phi: H^k(-1) \times H^{n-k}(-1) \to \mathbb R$. 
\end{example}

The above examples give rise to totally geodesic Lagrangian submanifolds of $\Q$. In Section \ref{sec:par} we will prove that they are essentially the only totally geodesic Lagrangian submanifolds of $\Q$ and that they are even the only minimal Lagrangian submanifolds of $\Q$ with parallel second fundamental form.

The following three families of examples are rotation hypersurfaces of $H^{n+1}_1(-1)$. Inspired by \cite{DoCarmoDajczer1983}, where rotation hypersurfaces in Riemannian real space forms were defined, eight types of rotation hypersurfaces in semi-Riemannian real space forms were introduced in \cite{MoruzVrancken2020}. When one requires that the rotation hypersurface is spacelike and the ambient space form is $H^{n+1}_1(-1)$, there are three types remaining, depending on the signature of the metric on the rotation axis. We start here from a description of these three families which is closer to \cite{DoCarmoDajczer1983} than to \cite{MoruzVrancken2020}.

\begin{example} \label{ex:3}
Let $\varphi(t_1,\dots, t_{n-1})=(\varphi_1(t_1,\dots, t_{n-1}),\dots, \varphi_n(t_1,\dots, t_{n-1}))$ be an orthogonal para\-metrisation of $H^{n-1}(-1)\subset \mathbb{R}^n_1$ and let $I \subset \mathbb{R} \rightarrow H^{n+1}_1(-1): s\mapsto (f(s), g(s), 0, \dots, 0, h(s))$ be a curve parametrized by arc length. This means that the real functions $f$, $g$ and $h$ satisfy $-f^2 -g^2+h^2=-1$ and $-(f')^2-(g')^2+(h')^2=1$. Then 
\begin{equation*}
a_3(s, t_1,\dots, t_{n-1}) = (f(s), g(s)\varphi(t_1,\dots, t_{n-1}), h(s)) 
\end{equation*}
parametrizes a rotation hypersurface of $H^{n+1}_1(-1)$ for which, in the notation of \cite{MoruzVrancken2020}, the axis of rotation $\Pi^2$ has signature $( 1 , -1)$. Note that $b_3 = (hg'-gh', (fh'-hf')\varphi, fg'-gf')$ defines a normal vector field to the hypersurface in $H^{n+1}_1(-1)$ satisfying $\langle b_3, b_3 \rangle_2=-1$. We can now compute the principal curvatures associated to this choice of unit normal using the same technique as in \cite{DoCarmoDajczer1983}: the coordinate vector fields $\{\partial_s,\partial_{t_1},\ldots,\partial_{t_{n-1}}\}$ form an orthogonal basis of principal vector fields at every point and the corresponding principal curvatures are respectively 
\begin{equation*}
\lambda_1= \frac{g'' - g}{\sqrt{1+(g')^2 - g^2}}, \quad \lambda_2 = \ldots = \lambda_{n}= -\frac{\sqrt{1+(g')^2 -g^2}}{g}.
\end{equation*}
	
The Gauss map of the hypersurface $a_3$, using the normal vector $b_3$, is given by
\begin{equation*}
G_3 = [(f+i(hg'-gh'), (g+i (fh'-hf'))\varphi, h +i(fg'-gf'))].
\end{equation*}
If we choose $A\in \mathcal{A}$ as in Choice \ref{ex:choiceA}, using the canonical lift of $G_3$, we obtain from \eqref{eq:choiceA} by a straightforward computation that
\begin{equation*}
A (dG_3) \partial_s = \frac{(g''-g)^2-(1-g^2+(g')^2)}{(g''-g)^2+ 1-g^2+(g')^2} (dG_3)\partial_s - \frac{2(g''-g)\sqrt{1-g^2+(g')^2}}{(g''-g)^2+1-g^2+(g')^2}J(dG_3)\partial_s.
\end{equation*}
The coefficients on the right hand side are equal to $\cos(2\theta_1)$ and $-\sin(2\theta_1)$, so they determine $\theta_1$ up to an integer multiple of $\pi$. In particular, we have
\begin{equation*}
\cot\theta_1 = \frac{1+\cos(2\theta_1)}{\sin(2\theta_1)} = \frac{g'' - g}{\sqrt{1+(g')^2 - g^2}} = \lambda_1.
\end{equation*}
%\begin{equation*}
%\theta_1= \frac{1}{2} \arcsin\left(\frac{2(g''-g)\sqrt{1-g^2+(g')^2}}{(g''-g)^2+1-g^2+(g')^2}\right) \mod \pi
%\end{equation*}
This is not a coincidence. Theorem \ref{theo:LagrangiansandGaussmaps} below states that for the particular choice of $A$ we have made here, one always has $\lambda_j = \cot\theta_j$.
In the present example, we can also compute
\begin{equation*}
A (dG_3) \partial_{t_i} = \frac{-2g^2+1+(g')^2}{1+(g')^2}(dG_3)\partial_{t_i} + \frac{2g\sqrt{1-g^2+(g')^2}}{1+(g')^2}J(dG_3)\partial_{t_i}
\end{equation*}
for $i \in \{1,\ldots,n-1\}$, from which it follows that
\begin{equation*}
\cot\theta_j = -\frac{\sqrt{1+(g')^2 -g^2}}{g} = \lambda_j
\end{equation*}
for $j \in \{2,\ldots,n\}$.
\end{example}

\begin{example} \label{ex:4}
Let $\varphi(t_1,\dots, t_{n-1})=(\varphi_1(t_1,\dots, t_{n-1}),\dots, \varphi_n(t_1,\dots, t_{n-1}))$ be an orthogonal para\-metrisation of $S^{n-1}(1)\subset \mathbb{R}^n$ and let $I \subset \mathbb{R} \rightarrow H^{n+1}_1(-1): s\mapsto (f(s), g(s), h(s), 0, \dots, 0)$ be a curve parametrized by arc length. Hence, the functions $f$, $g$ and $h$ satisfy $-f^2 -g^2+h^2=-1$ and $-(f')^2-(g')^2+(h')^2=1$. Then
\begin{equation*}
a_4(s, t_1,\dots, t_{n-1}) = (f(s), g(s), h(s)\varphi(t_1,\dots, t_{n-1}))
\end{equation*}
parametrizes a rotation hypersurface of $H^{n+1}_1(-1)$ for which, in the notation of \cite{MoruzVrancken2020}, the axis of rotation $\Pi^2$ has signature $( -1 , -1)$. Note that $b_4 = (hg'-gh', fh'-hf', (gf'-fg')\varphi)$ is a normal vector field to the hypersurface in $H^{n+1}_1(-1)$ for which $\langle b_4, b_4 \rangle_2=-1$. As in the previous example, we can compute the principal curvatures associated to $b_4$ to be
\begin{equation*}
\lambda_1 = \frac{h - h''}{\sqrt{(h')^2 - h^2-1}}, \quad \lambda_2 = \ldots = \lambda_{n}= \frac{\sqrt{(h')^2 -h^2-1}}{h}.
\end{equation*}
Also here, these correspond to the principal vector fields $\partial_s,\partial_{t_1},\ldots,\partial_{t_{n-1}}$ respectively.
	
The Gauss map of the hypersurface $a_4$ is given by
\begin{equation*}
G_4 = [(f+i(hg'-gh'), g+i (fh'-hf'), (h +i(fg'-gf'))\varphi)].
\end{equation*}
Choosing $A \in \mathcal A$ as in Choice \ref{ex:choiceA} for the canonical lift of $G_4$, we can compute 
\begin{align*}
& A (dG_4) \partial_s = \frac{(h-h'')^2 - ((h')^2-h^2-1)}{(h-h'')^2+(h')^2-h^2-1}(dG_4)\partial_s - \frac{2(h-h'') \sqrt{(h')^2-h^2-1}}{(h-h'')^2+(h')^2-h^2-1}J(dG_4)\partial_s,\\
& A (dG_4) \partial_{t_i} = \frac{(h')^2-2h^2-1}{(h')^2-1} (dG_4) \partial_{t_i} - \frac{2 h \sqrt{(h')^2-h^2-1}}{(h')^2-1} J(dG_4)\partial_{t_i}
\end{align*}
for $i \in \{1,\ldots,n-1\}$. As in the previous example, we obtain
\begin{equation*}
\cot\theta_1 = \frac{h - h''}{\sqrt{(h')^2 - h^2-1}} = \lambda_1, \quad \cot\theta_j = \frac{\sqrt{(h')^2 -h^2-1}}{h} = \lambda_j
\end{equation*}
for $j \in \{2,\ldots,n\}$.
\end{example}

\begin{example} \label{ex:5}
To describe the last example, we will use a different basis of $\mathbb{R}^{n+2}_2$. Denoting the standard basis by $\{e_1,\dots, e_{n+2}\}$, we define $u_1=e_1$, $u_2= e_3+e_2$, $u_3= e_3-e_2$ and $u_j=e_j$ for $j \in \{4,\ldots,n+2\}$. Note that $u_2$ and $u_3$ are null vectors satisfying $\langle u_2,u_3\rangle_2= 2$. Consider a curve  $I \subset \mathbb R \to H^{n+1}_1(-1): s \mapsto f(s)u_1+g(s)u_2+h(s)u_3$ parametrized by arc length. This means that the functions $f$, $g$ and $h$ satisfy $-f^2+4gh=-1$ and $-(f')^2+4g'h'=1$. Then
\begin{equation*}
a_5(s,t_1,\dots,t_{n-1}) = f(s)u_1 + \frac{f^2(s)\!-\!1\!-\!h^2(s) \sum_{i=1}^{n-1} t_i^2}{4h} u_2 + h(s) u_3 + h(s)t_1 u_4  + \ldots + h(s)t_{n-1} u_{n+2}
\end{equation*}
parametrizes a rotation hypersurface of $H^{n+1}_1(-1)$ for which, in the notation of \cite{MoruzVrancken2020}, the axis of rotation $\Pi^2$ has signature $(-1 , 0)$. We define the following vector field : 
\begin{multline*}
b_5 = - \sqrt{(h')^2-h^2}\Bigg( \frac{h'f'\!-\!hf}{(h')^2-h^2}u_1 + \frac{2hh'ff'\!+\!((h')^2\!-\! h^2)(1\!+\!f^2)\! -\! h^2((h')^2\!-\!h^2)\sum_{i=1}^{n-1}t_i^2)}{4h^2((h')^2\!-\!h^2)}u_2 \\ + u_3 + t_1u_4 + \ldots + t_{n-1}u_{n+2} \Bigg).
\end{multline*}
Using that $(hf'-fh')^2= (h')^2-h^2$, one can check that this is a normal vector field to the hypersurface, satisfying $\langle b_5,b_5 \rangle_2 = -1$. The associated principal curvatures are
\begin{equation*}
\lambda_1= \frac{h+h''}{\sqrt{(h')^2-h^2}}, \quad \lambda_2 = \ldots = \lambda_n= \frac{\sqrt{(h')^2 -h^2}}{h}.
\end{equation*}
As in the previous examples, these correspond to the principal vector fields $\partial_s,\partial_{t_1},\ldots,\partial_{t_{n-1}}$ respectively.	

The Gauss map of the hypersurface $a_5$ is given by $G_5 = [(a_5+i b_5)]$. If we choose $A \in \mathcal A$ as in Choice \ref{ex:choiceA} for the canonical lift of $G_5$, we can compute
\begin{align*}
& A (dG_5) \partial_s = \frac{(h+h'')^2-(h')^2+h^2}{(h+h'')^2+(h')^2-h^2}  (dG_5)\partial_s - \frac{2(h+h'')\sqrt{(h')^2-h^2}}{(h')^2-h^2+(h+h'')^2} J(dG_5) \partial_s, \\
& A (dG_5) \partial_{t_i} = \frac{(h')^2-2h^2}{(h')^2} (dG_5)\partial_{t_i} - \frac{2h\sqrt{(h')^2-h^2}}{(h')^2} J(dG_5) \partial_{t_i}
\end{align*}
for $i \in \{1,\ldots,n-1\}$. As in the previous examples, we find
\begin{equation}
\cot\theta_1 = \frac{h+h''}{\sqrt{(h')^2-h^2}} = \lambda_1, \quad \cot\theta_j= \frac{\sqrt{(h')^2 -h^2}}{h} = \lambda_j
\end{equation}
for $j \in \{2,\ldots,n\}$.
\end{example}

\subsection{Relation between principal curvatures and the angle functions of the Gauss map}

Knowing that the Gauss map of a hypersurface of $H^{n+1}_1(-1)$ is a Lagrangian immersion into $\Q$ raises the question whether, given a Lagrangian immersion $f:M^n\rightarrow \Q$, we can go back and find a spacelike hypersurface of $H^{n+1}_1(-1)$ with Gauss map $f$. It turns out that we can always do this, at least locally. A similar question for Lagrangian immersions into $Q^n$ and hypersurfaces of $S^n(1)$ was answered in \cite{VW}, and we can prove the following theorem using similar methods.

\begin{theorem}\label{theo:LagrangiansandGaussmaps}
Let $a: M^n \to H^{n+1}_1(-1)$ be a spacelike hypersurface with unit normal~$b$. Then the Gauss map $G: M^n \to \Q: p \mapsto [a(p)+ib(p)]$ is a Lagrangian immersion. Moreover, if $A$ is chosen as in Choice \ref{ex:choiceA} using the canonical horizontal lift 
	\begin{equation} \label{eq:Ghat}
	\widetilde G: M^n \to {V^*}^{2n+1}_1: p \mapsto \frac{1}{\sqrt 2}(a(p)+ib(p)),
	\end{equation}
then the relation between the principal curvatures $\lambda_1,\ldots,\lambda_n$ of $a$, with respect to the shape operator associated to $b$, and the angle functions $\theta_1,\ldots,\theta_n$ of $G$ is
\begin{equation} \label{theta_lambda_1} 
	\lambda_j = \cot\theta_j
\end{equation}
for $j=1,\ldots,n$.

Conversely, if $f: M^n \to \Q$ is a Lagrangian immersion, then for every point of $M^n$ there exist an open neighborhood $U$ of that point in $M^n$ and an immersion $a: U \to H^{n+1}_1(-1)$ with Gauss map $f|_U$. This immersion is not unique, nor are its principal curvature functions. However, for any choice of $a$, a local frame of principal directions for $a$ is adapted to $f$ in the sense that \eqref{eq:defBandC} holds for any choice of $A$ and the principal curvature functions $\lambda_1,\ldots,\lambda_n$ of $a$ are related to the corresponding local angle functions $\theta_1,\ldots,\theta_n$ by
\begin{equation} \label{eq:angles-princcurv}
\cot(\theta_j-\theta_k) = \pm \frac{\lambda_j \lambda_k + 1}{\lambda_j - \lambda_k} %CHECK THIS! ik denk ok.
\end{equation}
for $j,k=1,\ldots,n$ at points where $\lambda_j \neq \lambda_k$.
\end{theorem}

Let us revisit the examples given above. By comparing the principal curvatures and angle functions given in all the examples, we see that the relation \eqref{eq:angles-princcurv} is still valid if we allow the value $\infty$. Note also that \eqref{eq:angles-princcurv} holds for any two chosen indices. In Example \ref{ex:3}, Example \ref{ex:4} and Example \ref{ex:5} we have chosen $A \in \mathcal A$ as in Choice \ref{ex:choiceA} for the canonical horizontal lift of the Gauss map and we noticed already that \eqref{theta_lambda_1} holds.

\subsection{Relation between principal curvatures and the mean curvature of the Gauss map}

There is also a relation between the principal curvatures of a spacelike hypersurface of anti-de Sitter space and the mean curvature of its Gauss map, which is given by the following theorem. A similar result for hypersurfaces of $S^{n+1}(1)$ and their Gauss maps into $Q^n$ was proven in  \cite{Palmer1997}. Here we give a slightly different proof, but the main ideas are similar. 
	
\begin{theorem}\label{theo:Palmerformula}
Let $a:M^n \rightarrow H^{n+1}_1(-1)$ be an immersion and let $G: M^n \to \Q$ be the Gauss map of this immersion. Denote by $\lambda_1,\ldots,\lambda_n$ the eigenvalues of the shape operator of $a$, associated to the unit normal $b$ that was used to construct $G$.  If $H$ is the mean curvature vector of $G$, then 
\begin{equation}\label{eq:minimal}
g(JH,\cdot)=\frac{1}{n} d\left(\sum_{i=1}^n\arctan(\lambda_j)\right)=\frac{1}{n} d\left(\im\left(\log\left(\prod_{j=1}^n (1+i\lambda_j)\right)\right)\right).
\end{equation}
\end{theorem}

\begin{proof}
Let $\{\epsilon_1,\dots, \epsilon_n\}$ be an orthonormal frame on $M^n$ with respect to the metric induced by $a$, such that $S \epsilon_j = \lambda_j \epsilon_j$ for $j \in \{1,\ldots,n\}$, where $S$ is the shape operator associated to $b$. Then the vector fields $e_j=\tfrac{\sqrt{2}}{\sqrt{1+\lambda_j^2}}\epsilon_j$ form an orthonormal frame on $M^n$ with respect to the metric induced by the canonical lift $\tilde{G}$. If $D$ is the Euclidean connection on $\mathbb R^{2n+4}_2$, we can identify $e_j$ with $(d\widetilde{G})e_j = D_{e_j}\tilde{G}$, its image under the derivative of $\tilde{G}$. 
		
We compute $g(JH,X)$ for a tangent vector $X$ to $M^n$. Since $nH$ is the trace of the second fundamental form $h$ and $g(h(\cdot,\cdot),J\cdot)$ is totally symmetric, it follows that
\begin{equation}\label{eq:gJHX}
n g(JH,X) = \sum_{j=1}^n g(Jh(e_j,e_j), X) = -\sum_{j=1}^n g(\lccQ_X e_j, Je_j) = -\sum_{j=1}^n \left\langle\!\!\left\langle D_{X} D_{e_j} \tilde{G}, iD_{e_j} \tilde{G} \right\rangle\!\!\right\rangle_2.
\end{equation}
Since $e_j=\tfrac{\sqrt{2}}{\sqrt{1+\lambda_j^2}}\epsilon_j$, we get from \eqref{eq:LiftOfBasisUnderGaussMap} that
\begin{multline*}
\left\langle\!\!\left\langle D_{X} D_{e_j} \tilde{G}, iD_{e_j} \tilde{G} \right\rangle\!\!\right\rangle_2
=\frac{2}{1+\lambda_j^2} \left\langle\!\!\left\langle D_XD_{\epsilon_j}\tilde{G}, iD_{\epsilon_j}\tilde{G} \right\rangle\!\!\right\rangle_2
= \frac{1}{1+\lambda_j^2} \left\langle\!\left\langle D_X(1-i\lambda_j)\epsilon_j, (i+\lambda_j)\epsilon_j \right\rangle\!\right\rangle_2 \\
= \frac{1}{1+\lambda_j^2} \left\langle\!\left\langle -iX(\lambda_j)\epsilon_j, i\epsilon_j \right\rangle\!\right\rangle_2 
= -\frac{X(\lambda_j)}{1+\lambda_j^2}
= -X(\arctan(\lambda_j)) 
= -X(\im(\log(1+i\lambda_j))).
\end{multline*}
%Notice that $\arctan(\lambda_j)=\frac{i}{2}(\log(1-i\lambda_j)-\log(1+i\lambda_j)=\frac{i}{2}(\overline{\log(1+i\lambda_j)}-\log(1+i\lambda_j)=\frac{i}{2}(-2i\im(\log(1+i\lambda_j))=\im(\log(1+i\lambda_j))$.
The result follows by substituting this in \eqref{eq:gJHX}.
\end{proof}

The Gauss maps $G_1$ and $G_2$ are minimal immersions since the principal curvatures of the corresponding hypersurfaces of anti-de Sitter space are constant.

\section{Minimal Lagrangian submanifolds of $\Q$ with parallel second fundamental form} \label{sec:par}
%%%%%%%%%%%%%%%%%%%%%%%%%%%%%%%%%%%%%%%%%%%%%%%%%%%%%%%%%%%%%%%%%%%%%%%%%%%%%%%%%%%%%%%	

In this section we classify all minimal Lagrangian submanifolds of $\Q$ with parallel second fundamental form. In particular, we show that they are totally geodesic and essentially correspond to the Gauss maps from Example \ref{ex:1princcurv} and Example \ref{ex:2princcurv} above.

\begin{theorem} \label{theo:parallel_sff}
Let $f: M^n \to \Q$, with $n \geq 2$, be a minimal Lagrangian immersion with parallel second fundamental form. Then, up to isometries of $\Q$, $f$ is the Gauss map of a standard embedding of an open part of $H^n$ or of $H^k \times H^{n-k}$ into $H^{n+1}_1(-1)$. In particular, $f$ is totally geodesic.
\end{theorem}

\begin{proof}
Since $\bar\nabla h = 0$, the equation of Codazzi \eqref{eq:CodazziEqLagSubm} yields
$$ g(BY,Z)CX - g(BX,Z) CY - g(CY,Z) BX + g(CX,Z) BY = 0 $$
for all vector fields $X$, $Y$ and $Z$ on $M^n$. Denote by $\{e_1,\ldots,e_n\}$ a local orthonormal frame on $M^n$, diagonalizing $B$ and $C$ and corresponding to local angle functions $\theta_1,\ldots,\theta_n$. By choosing $X=Z=e_i$ and $Y=e_j$ for different indices $i$ and $j$, we obtain $\sin(2(\theta_i-\theta_j))=0$ and hence
$$ \forall i,j \in \{1,\ldots,n\}, \ \exists m_{ij} \in \mathbb Z : \ \theta_i-\theta_j = m_{ij} \frac{\pi}{2}. $$
It now follows from Theorem \ref{theo:LagrangiansandGaussmaps} that $f$ is the Gauss map of an immersion $a:M^n \to H^{n+1}_1(-1)$ whose principal curvature functions $\lambda_1,\ldots,\lambda_n$ satisfy
\begin{equation} \label{eq:parallel1}
\forall i,j \in \{1,\ldots,n\}: \ \lambda_i = \lambda_j \mbox{ or } \lambda_i \lambda_j + 1 = 0.
\end{equation}
On the other hand, the equation of Codazzi for the immersion $a$ yields that
\begin{equation} \label{eq:parallel2}
\forall i,j \in \{1,\ldots,n\}: \ e_i(\lambda_j)e_j - e_j(\lambda_i)e_i = 0.
\end{equation}
It follows from \eqref{eq:parallel1} that there are at most two different principal curvatures and, if there are two different ones, their product is equal to $-1$. Combining this with \eqref{eq:parallel2} yields that the principal curvatures are constant, i.e., that the hypersurface $a$ is isoparametric. It was proven in \cite{Lopez} that there are only two types of spacelike isoparametric hypersurfaces of anti-de Sitter space and it is not hard to see that these two families precisely correspond to Example~\ref{ex:1princcurv} and Example~\ref{ex:2princcurv}. We already checked that the Gauss maps in these examples are totally geodesic.
\end{proof} 

\begin{corollary}
Let $f: M^n \to \Q$, with $n \geq 2$, be a totally geodesic Lagrangian immersion. Then, up to isometries of $\Q$, $f$ is the Gauss map of a standard embedding of an open part of $H^n$ or of $H^k \times H^{n-k}$ into $H^{n+1}_1(-1)$.
\end{corollary}

%%%%%%%%%%%%%%%%%%%%%%%%%%%%%%%%%%%%%%%%%%%%%%%%%%%%%%%%%%%%%%%%%%%%%%%%%%%%%%%%%%%
\section{Minimal Lagrangian submanifolds of $\Q$ with constant sectional curvature} 
\label{sec:4}
%%%%%%%%%%%%%%%%%%%%%%%%%%%%%%%%%%%%%%%%%%%%%%%%%%%%%%%%%%%%%%%%%%%%%%%%%%%%%%%%%%%

We can compute the sectional curvature of a Lagrangian submanifold of $\Q$ from the equation of Gauss \eqref{eq:GaussEqLagSubm}. If $f: M^n \to \Q$ is a Lagrangian immersion and $\{e_1,\ldots,e_n\}$ is a local orthonormal frame on $M^n$ such that $Ae_i = \cos(2\theta_i)e_i - \sin(2\theta_i)Je_i$ for some $A \in \mathcal A$ and all $i \in \{1,\ldots,n\}$, then the sectional curvature of the metric on $M^n$ induced by $f$ is determined by
\begin{equation} \label{eq:sectionalcurvature} 
K(\mbox{span}\{e_i,e_j\}) = -2\cos^2(\theta_i-\theta_j) + g(h(e_i,e_i),h(e_j,e_j)) - g(h(e_i,e_j),h(e_i,e_j)).
\end{equation}
In particular, the Gauss map $G_1$ from Example \ref{ex:1princcurv} gives rise to a metric of constant sectional curvature $c=-2$ on $H^n(-1)$ and, in dimension $n=2$, the Gauss map $G_2$ from Example \ref{ex:2princcurv} gives rise to a metric of constant sectional curvature $c=0$ on $H^1(-1) \times H^1(-1)$.

In this section we prove that these two are essentially the only minimal Lagrangian submanifolds with constant sectional curvature of $\Q$.
	
\begin{theorem}\label{theo:csc}
Let $f:M^n\rightarrow \Q$, $n \geq 2$, be a minimal Lagrangian immersion such that $M^n$ has constant sectional curvature $c$. Then, up to isometries of $\Q$, $f$ is the Gauss map of a standard embedding of an open part of $H^n$ into $H^{n+1}_1(-1)$ or of $H^1 \times H^1$ into $H^3_1(-1)$. In the former case
$c = -2$ and in the latter case $c = 0$.
\end{theorem}
	
The proof of the theorem relies on several lemmas and propositions. We will only give the proofs if they are sufficiently different from their counterparts for minimal Lagrangian submanifolds with constant sectional curvature of $Q^n$, see \cite{LMVdVVW}.
	
\begin{lemma}\label{lem:EquationsConstSecCurv}
Let $f:M^n\rightarrow \Q$ be a Lagrangian immersion such that $M^n$ has constant sectional curvature. Assume an almost complex structure $A\in\mathcal{A}$ is fixed on $\Q$ and $\{e_1,\dots,e_n\}$ is an orthonormal frame on $M^n$ such that $Ae_i = \cos(2\theta_i)e_i - \sin(2\theta_i)Je_i$ for all $i \in \{1,\ldots,n\}$. Then
\begin{equation*}
\begin{aligned}
\sin(&\theta_i-\theta_k)\sin(2\theta_j-\theta_k-\theta_i)(h_{i\ell}^k Je_j+\delta_{j\ell}h(e_i,e_k))\\ 
&+ \sin(\theta_k-\theta_j)\sin(2\theta_i-\theta_j-\theta_k)(h_{j\ell}^kJe_i +\delta_{i\ell}h(e_k,e_j))\\
&+ \sin(\theta_j-\theta_i)\sin(2\theta_k-\theta_i-\theta_j)(h_{i\ell}^jJe_k+\delta_{k\ell}h(e_j,e_i))=0
\end{aligned}
\end{equation*}
for all $i,j,k,l \in \{1,\dots,n\}$. In particular,
\begin{align}
&h_{ii}^k \sin(\theta_k-\theta_i)\sin(\theta_i+\theta_k-2\theta_j)=h_{jj}^k\sin(\theta_k-\theta_j)\sin(\theta_j+\theta_k-2\theta_i),\label{eq:ConstSecCurv1}\\
&h_{ij}^k\sin(\theta_j-\theta_i)\sin(\theta_i+\theta_j-2\theta_k)=0 \nonumber
\end{align}
for mutually different $i,j,k \in \{1,\dots,n\}$ and
\begin{equation*}
h_{ij}^k\sin(\theta_i-\theta_j)\sin(\theta_i+\theta_j-2\theta_{\ell})=0
\end{equation*}
for mutually different $i,j,k,l=1,\dots,n$.
\end{lemma}
	
\begin{lemma}\label{prop:CSCAllAnglesDiffOrEqual}
For $n\geq 3$, let $f:M^n\rightarrow\Q$ be a minimal Lagrangian immersion such that $M^n$ has constant sectional curvature. Then the local angle functions are either all the same or all different modulo~$\pi$. In the former case, the immersion is the Gauss map of a standard embedding of an open part of $H^n$ into $H^{n+1}_1(-1)$.
\end{lemma}
	
\begin{lemma}\label{prop:CSCHijkZero}
Let $f:M^n\rightarrow \Q$ be a minimal Lagrangian immersion such that $M^n$ has constant sectional curvature. Then $h_{ij}^k=0$ for all mutually different indices $i,j$ and $k$. 
\end{lemma}
	
Theorem \ref{theo:csc} is proven by considering the dimensions $n=2$, $n=3$, $n=4$ and $n\geq 5$ separately. The only differences with respect to the situation of $Q^n$ occur for $n=2$ and $n=3$, so we will prove Propositions \ref{prop:CSCdim2} and \ref{prop:CSCdim3} below, but not Proposition \ref{prop:CSCdim>4} (which corresponds to two propositions in \cite{LMVdVVW}: one for $n=4$ and one for $n \geq 5$).
	
\subsection{Classification in dimension $n=2$}
	
\begin{proposition} \label{prop:CSCdim2}
If $f:M^2\rightarrow {Q^*}^2$ is a minimal Lagrangian immersion such that the induced metric on $M^2$ has constant Gaussian curvature, then either the induced metric has Gaussian curvature $-2$ and $f$ is the Gauss map of a part of a standard embedding $H^2 \rightarrow H^3_1(-1)$, or the induced metric is flat and $f$ is the Gauss map of a part of a standard embedding $H^1 \times H^1 \rightarrow H^3_1(-1)$.
\end{proposition}
	
\begin{proof}
Since ${Q^*}^2$ is isometric to $H^2(-4)\times H^2(-4)$, the result follows from the classification of minimal Lagrangian surfaces with constant Gaussian curvature in $H^2(c)\times H^2(c)$, given in \cite{GVWX}.
\end{proof}
	
\subsection{Classification in dimension $n=3$}
	
\begin{proposition} \label{prop:CSCdim3}
Let $f:M^3\rightarrow {Q^*}^3$ be a Lagrangian minimal immersion with constant sectional curvature, then $M^3$ has constant sectional curvature $-2$ and $f$ is the Gauss map of part of a standard embedding $H^3(c)\rightarrow H^4_1(-1)$.
\end{proposition}
	
	\begin{proof}
		Choose $A\in \mathcal{A}$ as in Choice \ref{ex:choiceA2}. From Lemma~\ref{prop:CSCAllAnglesDiffOrEqual} and Lemma~\ref{prop:CSCHijkZero}, it follows that we only need to consider the case where all local angle functions are different modulo $\pi$ and $h_{12}^3=0$. 
		
		We denote
		\begin{align*}
		&x= \sin(\theta_1-\theta_2)\sin(\theta_1+\theta_2-2\theta_3),\\
		&y= \sin(\theta_2-\theta_3)\sin(\theta_2+\theta_3-2\theta_1),\\
		&z=\sin(\theta_3-\theta_1)\sin(\theta_3+\theta_1-2\theta_2).
		\end{align*}
		It follows from simple trigonometric identities that $x+y+z=0$. Moreover, \eqref{eq:ConstSecCurv1} is equivalent to
		\begin{equation}\label{eq:threefromprop}
		h_{22}^1x+h_{33}^1z=0,\quad h_{11}^2x+h_{33}^2y=0,\quad h_{11}^3z+h_{22}^3y=0.
		\end{equation}
		We will distinguish three cases.
		
		\smallskip
		
		\textit{Case 1: At least two of the local functions $x,y$ and $z$ are zero.} In this case all 3 must vanish since $x+y+z=0$. Since the local angle functions are mutually different modulo $\pi$, we may assume $\theta_1<\theta_2<\theta_3$. It follows from $x=y=z=0$ that $3\theta_1=3\theta_2=3\theta_3 \mod \pi$. Therefore, $\theta_1,\theta_2$ and $\theta_3$ are all constant. This implies that the principal curvatures of the corresponding spacelike hypersurface in anti-de Sitter are also constant, which contradicts the result of \cite{Lopez}.
		
		\smallskip
		
		\textit{Case 2: Exactly one of the local functions $x,y$ and $z$ is zero.} Without loss of generality, we may assume that $x=0$, and $y=-z\neq0$. In follows from \eqref{eq:threefromprop} that $h_{33}^1=h_{33}^2=0$ and $h_{11}^3=h_{22}^3$. Since the angle functions are all different, the assumption $x=0$ implies that $\theta_1+\theta_2-2\theta_3=0\mod \pi$. Deriving this expression and using \eqref{eq:DerivativeOfTheta} gives $h_{11}^i+ h_{22}^i-2h_{33}^i=0$ for all $i\in\{1,2,3\}$. By the minimality condition, we also have that $h_{11}^i+h_{22}^i+h_{33}^i=0$, thus we get $h_{33}^i=0$ and $h_{11}^i+h_{22}^i=0$ for all $i\in\{1,2,3\}$. Therefore, the only possibly non-zero components of the second fundamental form are $h_{11}^1=-h_{22}^1$ and $h_{11}^2=-h_{22}^2$. By \eqref{eq:sectionalcurvature}, the sectional curvatures of the planes spanned by $\{e_1,e_2\}, \{e_1,e_3\}$ and $\{e_2,e_3\}$ are $K_{12}=-2\cos^2(\theta_2-\theta_1)-2(h_{11}^1)^2-2(h_{22}^2)^2, K_{13}=-2\cos^2(\theta_3-\theta_1)$ and $K_{23}=-2\cos(\theta_3-\theta_2)$. It follows from the latter two that $\theta_3-\theta_1$ and $\theta_3-\theta_2$ are constant. Deriving this  and using \eqref{eq:DerivativeOfTheta} gives $h_{11}^i=0$ and $h_{22}^i=0$. We conclude that the submanifold is totally geodesic, but it follows from $K_{12}=K_{13}$ that $y=0$, contradicting the assumption. 
		
		\smallskip
		
		\textit{Case 3: None of the local functions $x,y$ and $z$ are zero.} We work on an open subset of $M^n$ where none of the functions vanish. It follows from \eqref{eq:ConstSecCurv1} and the minimality condition that there are local functions $\alpha_1,\alpha_2$ and $\alpha_3$ on this subset such that 
		\begin{equation}
		\centering
		\begin{tabular}{l l l}
		$h_{11}^1=\alpha_1(z-x),$ & $h_{11}^2 =-\alpha_2 y,$ & $h_{11}^3=-\alpha_3 y$,\\
		$h_{22}^1=-\alpha_1 z,$ & $h_{22}^2=\alpha_2(y-x),$ & $h_{22}^3=\alpha_3 z$,\\
		$h_{33}^1=\alpha_1 x,$ & $h_{33}^2=\alpha_2 x,$ & $h_{33}^3=\alpha_3 (y-z).$\\
		\end{tabular}
		\end{equation}
		By Corollary \ref{cor:MinimalLagrangianImmersion} and equation (\ref{eq:DerivativeOfTheta}) we find the derivatives of the angle functions. Substituting these in the Codazzi equation \eqref{eq:CodazziEqLagSubm} for $(X,Y,Z)=(e_1,e_2,e_3)$ and using $\theta_1+\theta_2+\theta_3=0\mod\pi$ to eliminate $\theta_3$, we obtain
		\begin{equation}\label{eq:DerivativeAlphas}
		\begin{aligned}
		&e_1(\alpha_3)=	\frac{1}{4} \alpha_1 \alpha_3 \csc(2 \theta_1 + \theta_2)( \cos(4 \theta_1 + 5 \theta_2) + 7 \cos(3 \theta_2) - 5 \cos(2 \theta_1 + \theta_2)\\
		&\phantom{e_1(\alpha_3)=	\frac{1}{4}} - \cos(3 (2 \theta_1 + \theta_2)) -2 \cos(4\theta_1 - \theta_2) ), \\
		&e_2(\alpha_3)=\frac{1}{4} \alpha_2 \alpha_3 \csc(\theta_1 + 2 \theta_2)(-\cos(5 \theta_1 + 4 \theta_2)-7 \cos(3 \theta_1)  + 5 \cos(\theta_1 + 2 \theta_2) \\
		& \phantom{e_1(\alpha_3)=	\frac{1}{4}} +\cos(3 (\theta_1 + 2 \theta_2)) + 2 \cos(\theta_1 - 4 \theta_2) ).
		\end{aligned}
		\end{equation}
		Similarly, for $(X,Y,Z)=(e_2,e_3,e_1)$ we obtain		
		\begin{equation}
		\begin{aligned}
		&e_3(\alpha_1)=\frac{1}{4} \alpha_1 \alpha_3\csc(2 \theta_1 + \theta_2) (-\cos(4 \theta_1 - \theta_2) - 7 \cos(3 \theta_2) + 5 \cos(2 \theta_1 + \theta_2) \\
		&\phantom{e_3(\alpha_1)=\frac{1}{4}}  + \cos(6 \theta_1 + 3 \theta_2) + 2 \cos(4 \theta_1 + 5 \theta_2)) ,\\
		&e_2(\alpha_1)=-\frac{1}{4} \alpha_1 \alpha_2 \csc(\theta_1 - \theta_2)(\cos(3 \theta_1 - 3 \theta_2)- 7 \cos(3 (\theta_1 + \theta_2)) + 5 \cos(\theta_1 - \theta_2)  \\
		&\phantom{e_3(\alpha_1)=\frac{1}{4}} -	\cos(5 \theta_1 + \theta_2)+ 	2 \cos(\theta_1 + 5 \theta_2)) 
		\end{aligned}
		\end{equation}
		and for $(X,Y,Z)=(e_3,e_1,e_2)$ we obtain
		\begin{equation}
		\begin{aligned}
		&e_3(\alpha_2)=\frac{1}{4} \alpha_2 \alpha_3 \csc(\theta_1 + 2 \theta_2)( - 
		\cos(\theta_1 - 4 \theta_2)-7 \cos(3 \theta_1) + 5 \cos(\theta_1 + 2 \theta_2)  \\
		&\phantom{e_3(\alpha_2)=\frac{1}{4}} + 
		\cos(3 (\theta_1 + 2 \theta_2))+ 2 \cos(5 \theta_1 + 4 \theta_2)) ,\\
		&e_1(\alpha_2)=-\frac{1}{4} \alpha_1 \alpha_2 \csc(\theta_1 - \theta_2)( -\cos(\theta_1 + 5 \theta_2) - 7 \cos(3 (\theta_1 + \theta_2))+5 \cos(\theta_1 - \theta_2) \\
		&\phantom{e_3(\alpha_2)=\frac{1}{4}}+\cos(3 (\theta_1 - \theta_2))  + 2 \cos(5 \theta_1 + \theta_2)) .
		\end{aligned}
		\end{equation}
		Using these expressions for the derivatives of $\alpha_1,\alpha_2$ and $\alpha_3$, the $Je_1$-component of the Codazzi equation for $(X,Y,Z)=(e_2,e_1,e_1)$ yields $\alpha_1\alpha_2=0$, the $Je_2$-component of the Codazzi equation for $(X,Y,Z)=(e_3,e_2,e_2)$ yields $\alpha_2\alpha_3=0$ and the $Je_3$-component of the Codazzi equation for $(X,Y,Z)=(e_1,e_3,e_3)$ yields $\alpha_1\alpha_3=0$. This implies that at least two of the functions $\{\alpha_1,\alpha_2,\alpha_3\}$ vanish identically. Because of the symmetry of the problem, we may assume without loss of generality that $\alpha_1=\alpha_2=0$. 
		
		The $Je_2$-component of the Codazzi equation for $(X,Y,Z)=(e_1,e_2,e_1)$ then gives
		\begin{equation}
		\alpha_3^2 = 	2 \cos(\theta_1 - \theta_2) \csc(3 \theta_1) \csc(3 \theta_2),\label{eq:GammaKwadraat}
		\end{equation}
		whereas the $Je_2$-component of the Codazzi equation for $(X,Y,Z)=(e_3,e_1,e_3)$ gives
		\begin{align}
		e_3(\alpha_3)=&\frac{1}{16} \csc(3 \theta_1) \csc(2 \theta_1 + \theta_2) \csc(\theta_1 + 
		2 \theta_2) [\alpha_3^2 \cos(2 \theta_1 - 5 \theta_2) + 4 \alpha_3^2 \cos(6 \theta_1 - 3 \theta_2) \label{eq:DerivativeGammaToE3}\\
		&- 	32 \cos(2 \theta_1 + \theta_2) \sin(2 \theta_1 + \theta_2)^2 + \alpha_3^2 [\cos(8 \theta_1 + \theta_2)-16 \cos(
		4 \theta_1 - \theta_2) \nonumber\\
		& + 15 \cos(2 \theta_1 + \theta_2) + 4 \cos(3 (2 \theta_1 + \theta_2)) - \cos(5 (2 \theta_1 + \theta_2)) - 8 \cos(3 \theta_2) \nonumber\\
		&+ 	2 \sin(3 \theta_1) \sin(5 \theta_1 + 7 \theta_2)]].\nonumber
		\end{align}
		Substituting \eqref{eq:GammaKwadraat} and \eqref{eq:DerivativeGammaToE3} in the Codazzi equation for $(X,Y,Z)=(e_3,e_2,e_3)$ gives
		\begin{equation*}
		-5 \cos(\theta_1 - \theta_2) + 
		2 \cos(3 (\theta_1 - \theta_2)) + (1 + 
		2 \cos(2 (\theta_1 - \theta_2))) \cos(3\theta_1 + 3\theta_2 ) =0,\label{eq:EquationFullOfThetas}
		\end{equation*}
		while substituting them into the derivative of the Gauss equation for $(X,Y,Z,W)=(e_1,e_2,e_2,e_1)$ in the direction of $e_3$ yields
		\begin{equation*}
		1 + 2 \cos(2 (\theta_1 - \theta_2))=0.
		\end{equation*}
		Combining these last two equations, we obtain that all angle functions are constant, again contradicting the result in \cite{Lopez}.
	\end{proof}

\subsection{Classification in dimension $n\geq 4$}
	
\begin{proposition} \label{prop:CSCdim>4}
For $n\geq 4$, let $f:M^n\rightarrow \Q$ be a minimal Lagrangian immersion such that the induced metric on $M^n$ has constant sectional curvature. Then $f$ is the Gauss map of a part of a standard embedding of $H^n\rightarrow H^{n+1}_1(-1)$.
\end{proposition}
	
Combining Propositions \ref{prop:CSCdim2}, \ref{prop:CSCdim3} and \ref{prop:CSCdim>4} finishes the proof of Theorem \ref{theo:csc}.
	
	\bibliographystyle{amsplain}
	\bibliography{citations}
	
\end{document}